\documentclass[10pt]{amsart}

\usepackage[mathscr]{euscript}
\DeclareFontFamily{OT1}{rsfs}{}
\DeclareFontShape{OT1}{rsfs}{n}{it}{<-> rsfs10}{}
\DeclareMathAlphabet{\curly}{OT1}{rsfs}{n}{it}

\makeatletter
\newcommand{\eqnum}{\refstepcounter{equation}\textup{\tagform@{\theequation}}}
\makeatother

\newcommand\beq[1]{\begin{equation}\label{#1}}
	\newcommand\eeq{\end{equation}}
\newcommand\beqa{\begin{eqnarray*}}
	\newcommand\eeqa{\end{eqnarray*}}

\title[Derived categories of Quot schemes]{Derived categories of Quot schemes of locally free quotients via 
categorified Hall products}
\date{}
\author{Yukinobu Toda}

\usepackage{amscd}
\usepackage{amsmath}
\usepackage{amssymb}
\usepackage{amsthm}
\usepackage{float}
\usepackage{graphicx}
\usepackage{tikz}
\usepackage{xypic}
\usetikzlibrary{cd}
\usepackage{here}
\usetikzlibrary{intersections, calc, arrows.meta}

\usepackage{youngtab}

\usepackage{array}
\usepackage{amscd}
\usepackage[all]{xy}

\usepackage{tabulary}
\usepackage{booktabs}

\usepackage{amssymb, paralist, xspace, url, amscd, euscript, mathrsfs, stmaryrd}

\DeclareFontFamily{U}{rsfs}{%
	\skewchar\font127}
\DeclareFontShape{U}{rsfs}{m}{n}{%
	<-6>rsfs5<6-8.5>rsfs7<8.5->rsfs10}{}
\DeclareSymbolFont{rsfs}{U}{rsfs}{m}{n}
\DeclareSymbolFontAlphabet
{\mathrsfs}{rsfs}
\DeclareRobustCommand*\rsfs{%
	\@fontswitch\relax\mathrsfs}

\theoremstyle{plain}
\newtheorem{thm}{Theorem}[section]
\newtheorem{prop}[thm]{Proposition}
\newtheorem{lem}[thm]{Lemma}

\newtheorem{rmk}[thm]{Remark}
\newtheorem{cor}[thm]{Corollary}

\newtheorem{prop-defi}[thm]{Proposition-Definition}
\newtheorem{thm-defi}[thm]{Theorem-Definition}
\newtheorem{lem-defi}[thm]{Lemma-Definition}

\newtheorem{exam}[thm]{Example}

\newcommand{\ssslash}{/\!\!/}

\newcommand{\aA}{\mathcal{A}}

\newcommand{\cC}{\mathcal{C}}

\newcommand{\eE}{\mathscr{E}}
\newcommand{\fF}{\mathcal{F}}
\newcommand{\gG}{\mathscr{G}}
\newcommand{\hH}{\mathcal{H}}

\newcommand{\mM}{\mathcal{M}}

\newcommand{\oO}{\mathcal{O}}
\newcommand{\pP}{\mathscr{P}}
\newcommand{\qQ}{\mathcal{Q}}
\newcommand{\rR}{\mathcal{R}}

\newcommand{\xX}{\mathcal{X}}
\newcommand{\yY}{\mathcal{Y}}
\newcommand{\zZ}{\mathcal{Z}}

\newcommand{\fM}{\mathfrak{M}}

\newcommand{\Hom}{\mathop{\rm Hom}\nolimits}

\newcommand{\dR}{\mathbf{R}}

\newcommand{\Hilb}{\mathop{\rm Hilb}\nolimits}

\newcommand{\Pic}{\mathop{\rm Pic}\nolimits}

\newcommand{\Spec}{\mathop{\rm Spec}\nolimits}
\newcommand{\rank}{\mathop{\rm rank}\nolimits}

\newcommand{\sss}{\mathchar`-\rm{ss}}

\newcommand{\cneq}{\mathrel{\raise.095ex\hbox{:}\mkern-4.2mu=}}
\newcommand{\eqcn}{\mathrel{=\mkern-4.5mu\raise.095ex\hbox{:}}}

\newcommand{\Sym}{\mathop{\rm Sym}\nolimits}

\newcommand{\Quot}{\mathop{\rm Quot}\nolimits}

\newcommand{\RHom}{\mathop{\dR\mathrm{Hom}}\nolimits}

\newcommand{\GL}{\mathop{\rm GL}\nolimits}
\newcommand{\PGL}{\mathop{\rm PGL}\nolimits}

\newcommand{\MF}{\mathop{\rm MF}\nolimits}

\newcommand{\Crit}{\mathop{\rm Crit}\nolimits}

\newcommand{\wt}{\mathrm{wt}}

\newcommand{\inclusion}{\ar@<-0.3ex>@{^{(}->}[r]}
\newcommand{\upinclusion}{\ar@<-0.3ex>@{^{(}->}[u]}
\newcommand{\leinclusion}{\ar@<-0.3ex>@{^{(}->}[l]}
\newcommand{\doinclusion}{\ar@<-0.3ex>@{^{(}->}[d]}
\newcommand{\diasquare}{\ar@{}[rd]|\square}

\newcommand{\C}{\mathbb{C}^{\ast}}

\makeatletter
\renewcommand{\theequation}{%
	\thesection.\arabic{equation}}
\@addtoreset{equation}{section}
\makeatother

\begin{document}
	
	\begin{abstract}
	We prove Qingyuan Jiang's conjecture on 
	semiorthogonal decompositions of derived categories of 
	Quot schemes of locally free quotients. The
	author's result on categorified Hall products 
	for Grassmannian flips
 is applied to prove the conjecture. 
	\end{abstract}
	
	\maketitle
	
	%\setcounter{tocdepth}{1}
	%\tableofcontents
	
	\section{Introduction}
	\subsection{Quot formula}
	Let $X$ be a smooth quasi-projective variety over $\mathbb{C}$, 
	$\mathscr{G}$ a coherent sheaf on $X$ and $d \ge 0$ be an integer. 
	The relative Quot scheme 
	\begin{align}\label{quot}
		\Quot_{X, d}(\mathscr{G}) \to X
		\end{align}
	parametrizes rank $d$ locally free quotients of $\mathscr{G}$. 
	All the fibers of the above morphism are Grassmannian varieties, 
	whose dimensions are different in general. 
	Here we remark that $\Quot_{X, 0}(\mathscr{G})=X$. 
	
	Let us take a right exact sequence 
	\begin{align}\label{seq:right}
		\mathscr{E}^{-1} \stackrel{\phi}{\to} \mathscr{E}^0 \to \mathscr{G} \to 0
		\end{align}
	where $\mathscr{E}^0$ and $\mathscr{E}^{-1}$ are locally 
	free sheaves on $X$. 
	Let $\delta:=\rank(\mathscr{E}^0)-\rank(\mathscr{E}^{-1})$. 
	By taking its dual, we obtain the 
	right exact sequence 
	\begin{align*}
		\mathscr{E}_0 \stackrel{\phi^{\vee}}{\to}
		\mathscr{E}_1 \to \mathscr{H} \to 0
		\end{align*}
	where $\mathcal{E}_i:=(\mathscr{E}^{-i})^{\vee}$ and 
	$\mathscr{H}$ is the cokernel of $\phi^{\vee}$. 
	Note that $\rank(\mathcal{E}_1)-\rank(\mathcal{E}_0)=-\delta$. 
		As a dual side of (\ref{quot}), we also consider the relative Quot scheme 
	$\Quot_{X, d}(\mathscr{H}) \to X$. 
	
We will see that there exist 
quasi-smooth derived schemes over $X$
(see Section~\ref{subsec:derived})
\begin{align}\label{intro:dquot}
	\mathbf{Quot}_{X, d}(\mathscr{G})\to X
	\leftarrow 
	\mathbf{Quot}_{X, d}(\mathscr{H})
	\end{align}
which depend on a sequence (\ref{seq:right}) and 
with classical truncations 
 $\Quot_{X, d}(\mathscr{G})$, 
$\Quot_{X, d}(\mathscr{H})$
that have virtual dimensions 
 $\dim X+\delta d-d^2$, 
$\dim X-\delta d-d^2$ respectively. 
The following is the main result in this paper: 
\begin{thm}\emph{(Theorem~\ref{cor:sod})}\label{intro:thm}
	Suppose that $\delta \ge 0$. 
	There is a semiorthogonal decomposition of the form
	\begin{align*}
		D^b(\mathbf{Quot}_{X, d}(\mathscr{G}))=
		\left\langle \binom{\delta}{i}\mbox{-copies of }
		D^b(\mathbf{Quot}_{X, d-i}(\mathscr{H})) : 
		0\le i\le \mathrm{min}\{d, \delta\}  \right\rangle. 		
		\end{align*}
	\end{thm} 
The above result is 
a generalization of the 
conjecture by Qingyuan Jiang~\cite[Conjecture~A.5]{JiangQuot}
when $X$ is smooth
(see Corollary~\ref{cor:jiang}).
The case of $d=1$ is called \textit{projectivization formula} and 
proved in~\cite[Theorem~5.5]{Kuzhom}, \cite[Theorem~3.4]{Jiangproj}, \cite[Theorem~4.6.11]{TocatDT}. 
The 
$d=2$ case is proved in~\cite[Theorem~6.19]{JiangQuot}. 
The Quot formula in Theorem~\ref{intro:thm}
recovers 
several known formulas (see~\cite[Section~1.4.2]{JiangQuot} for details), 
e.g. 
Kapranov exceptional collection for Grassmannian~\cite{Kapranov}
(by setting $X$ to be a point), 
the projectivization formula~\cite{Kuzhom, Jiangproj, TocatDT}
(by setting $d=1$). 
The proof involves 
semiorthogonal 
decomposition of Grassmannian flip~\cite{BNFV, Toconi}, 
which itself generalizes Bondal-Orlov standard flip formula~\cite{B-O2}. 
	
Suppose that $\mathscr{G}$ has homological dimension 
less than or equal to one.
Then there is a sequence (\ref{seq:right}) so that $\phi$ is injective, 
and in 
that case $\mathscr{H}=\mathcal{E}xt^1_{\oO_X}(\mathscr{G}, \oO_X)$
(which is independent of a choice of (\ref{seq:right}) with $\phi$ injective), 
and $\delta=\rank(\mathscr{G})$. 
	In~\cite[Conjecture~A.5]{JiangQuot}, the conjecture 
is stated for derived categories of the 
classical Quot schemes $\Quot_{X, d}(\mathscr{G})$, 
$\Quot_{X, d}(\mathscr{H})$, 
when $\mathscr{G}$ has homological dimension less than or 
equal to one, $\mathscr{H}=\mathcal{E}xt^1_{\mathcal{O}_X}(\mathscr{G}, \mathcal{O}_X)$, and 
under some Tor-independence condition. 
The Tor-independence condition 
implies that
the dimensions of the above classical Quot schemes coincide with
the virtual dimensions if they are non-empty (see~\cite[Lemma~6.7]{JiangQuot}).
So in this case, they are equivalent to 
$\mathbf{Quot}_{X, d}(\mathscr{G})$, 
$\mathbf{Quot}_{X, d}(\mathscr{H})$ respectively, where 
we take a sequence (\ref{seq:right}) so that 
$\phi$ is injective. 
	We also note that, if $\Quot_{X, d}(\mathscr{G})=\emptyset$
then $\mathbf{Quot}_{X, d}(\mathscr{G})$ is equivalent to $\emptyset$ 
regardless of the virtual dimension, and the same is true for 
$\Quot_{X, d}(\mathscr{H})$. 
Therefore we obtain the following corollary, which 
proves~\cite[Conjecture~A.5]{JiangQuot} when $X$ is smooth: 

\begin{cor}\label{cor:jiang}
	Suppose that $\mathscr{G}$ has homological dimension less
	than or equal to one, and 
	$\mathscr{H}:=\mathcal{E}xt^1_{\oO_X}(\mathscr{G}, \oO_X)$. 
	Assume that $\dim \Quot_{X, d}(\mathscr{G})=\dim X+\delta d-d^2$
and $\dim \Quot_{X, d}(\mathscr{H})=\dim X-\delta d-d^2$, where $\delta=\rank(\mathscr{G}) \ge 0$. 
Then we have a semiorthogonal decomposition of the form 
\begin{align*}
	D^b(\Quot_{X, d}(\mathscr{G}))=
\left\langle \binom{\delta}{i}\mbox{-copies of }
D^b(\Quot_{X, d-i}(\mathscr{H})) : 
0\le i\le \mathrm{min}\{d, \delta\}  \right\rangle. 
\end{align*}	
	\end{cor}

\begin{exam}\label{exam:locfree}
	When $\mathscr{G}$ is locally free, then 
	$\mathscr{H}=0$.
	Suppose that $\delta \ge d$. Then 
	$\mathrm{Quot}_{X, d}(\mathscr{G})$ is a Grassmannian bundle over $X$ 
	with fiber $\mathrm{Gr}(d, \delta)$, 
	and $\mathrm{Quot}_{X, d}(\mathscr{H})=X$ for $d=0$, $\emptyset$ for $d>0$. 
	In this case, Corollary~\ref{cor:jiang}
	gives 
	\begin{align*}
		D^b(\Quot_{X, d}(\mathscr{G}))=
		\left\langle \binom{\delta}{d}\mbox{-copies of }
		D^b(X) \right\rangle. 		
	\end{align*}
	When $X$ is a point, the above semiorthogonal decomposition gives 
	Kapranov exceptional collection of Grassmannian variety~\cite{Kapranov}. 	 
\end{exam}

\begin{rmk}
	In~\cite[Conjecture~A.5]{JiangQuot}, the conjecture 
	is formulated in a more general assumption on $X$. We focus 
	on the case that $X$ is a smooth quasi-projective variety over $\mathbb{C}$
	in order to avoid some technical subtleties. 
	This assumption is enough for applications in~\cite[Section~1.5]{JiangQuot}. 
	\end{rmk}

\begin{rmk}
	Each fully-faithful functor 
	$D^b(\mathbf{Quot}_{X, d-i}(\mathscr{H})) \hookrightarrow 
	D^b(\mathbf{Quot}_{X, d}(\gG))$ in Theorem~\ref{intro:thm} can be 
	shown to be of Fourier-Mukai type, though 
	we will not discuss its details. However
	the proof of Theorem~\ref{intro:thm} does not give any 
	information about the kernel objects. 	
	\end{rmk}

We prove Theorem~\ref{intro:thm} by interpreting 
$(-1)$-shifted cotangent  
derived schemes in (\ref{intro:dquot}) (see Section~\ref{subsec:cotan}
for $(-1)$-shifted cotangent derived schemes or stacks) as 
d-critical Grassmannian flip in the sense of~\cite{Toddbir}
(see Remark~\ref{rmk:dcrit}), 
and then use Koszul duality together with 
categorified Hall products 
for families of Grassmannian flips. 
The 
categorified Hall products for 
Grassmannian flip are
used in~\cite{Toconi} 
as an intermediate 
step toward the categorification of 
wall-crossing formula of Donaldson-Thomas invariants on the 
resolved conifold. 

\subsection{Applications}
The Quot formula in Corollary~\ref{cor:jiang}
has lots of applications on derived categories of classical 
moduli spaces (see~\cite[Section~1.5]{JiangQuot}). 
Here we mention two examples:
one is a generalization of~\cite[Corollary~5.11]{TodDK} and~\cite[Corollary~1.3]{JiangQuot}
on semiorthogonal decompositions of 
varieties associated with Brill-Noether loci for curves, 
and the other one is a categorical blow-up formula of 
Hilbert schemes of points on surfaces obtained by Koseki~\cite{Kosekiup}. 

Let $C$ be a smooth projective curve over $\mathbb{C}$
with genus $g$. 
We denote by $\Pic^d(C)$ the Picard variety parameterizing degree $d$
line bundles on $C$, which is a $g$-dimensional complex torus and (non-canonically)
isomorphic to the Jacobian $\mathrm{Jac}(C)$ of $C$. 
The Brill-Noether locus on $\Pic^d(C)$ is defined by 
\begin{align*}
	W_d^r(C) \cneq \{L \in \Pic^d(C) : h^0(L) \ge r+1\}. 
	\end{align*}
There is a scheme $G_d^r(C)$ parameterizing $g_d^r$'s which appears in 
the classical study of Brill-Noether loci (see~\cite[Chapter 4, Section~3]{MR770932}). 
It is set theoretically given by 
\begin{align*}
	G_d^r(C) =\{(L, W) : L \in W_d^r(C), W \subset H^0(C, L), 
	\dim W=r+1\}
	\end{align*}
where $W$ is a vector subspace. 
If $C$ is a general curve, then 
$G_d^r(C)$ is a smooth projective 
variety of expected dimension $g-(r+1)(g-d+r)$. 
As explained in~\cite[Section~1.5.1]{JiangQuot},
for any $\delta \ge 0$ 
there is a coherent sheaf $\gG$ on $X=\Pic^{g-1+\delta}(C)$ 
of rank $\delta$ that has 
homological dimension less than or equal to 1 and 
such that 
\begin{align*}
	\Quot_{X, r+1}(\gG)=
	G_{g-1+\delta}^r(C), \ 
	\Quot_{X, r+1}(\mathscr{H})=G_{g-1-\delta}^r(C). 
	\end{align*}
Here $\mathscr{H}=\eE xt^1_{\oO_X}(\mathscr{G}, \oO_X)$. 
By applying Corollary~\ref{cor:jiang}, we have the following: 
\begin{cor}\label{cor:appl}
	Let $C$ be a general smooth projective curve 
	with genus $g$. 
	Then for any $r\in \mathbb{Z}_{\ge 0}$ and $\delta \ge 0$, there is a 
	semiorthogonal decomposition 
	\begin{align*}
		D^b(G_{g-1+\delta}^r(C))=
		\left \langle \binom{\delta}{i}\mbox{-copies of }
		D^b(G_{g-1-\delta}^{r-i}(C))  : 
		0\le i \le \mathrm{min}\{\delta, r+1 \} \right \rangle. 
		\end{align*}
	Here for $i=r+1$, we have
	$G_{g-1-\delta}^{-1}(C)=\Pic^{g-1-\delta}(C)$. 
	\end{cor}
The case of $r=0$ gives the semiorthogonal decomposition of 
symmetric products
\begin{align*}
	D^b(\mathrm{Sym}^{g-1+\delta}(C))=
	\left\langle D^b(\mathrm{Sym}^{g-1-\delta}(C)), 
	\delta \mbox{-copies of }D^b(\mathrm{Jac}(C)) \right\rangle 
	\end{align*}
proved in~\cite[Corollary~5.11]{TodDK}. 
The case of $r=1$ is given in~\cite[Corollary~1.3]{JiangQuot}, 
and it is 
\begin{align*}
	D^b(G_{g-1+\delta}^1(C))=
	\left\langle 
	D^b(G_{g-1-\delta}^1(C)), \delta \mbox{-copies of }
	D^b(\mathrm{Sym}^{g-1-\delta}(C)), 
	\binom{\delta}{2}\mbox{-copies of }
	D^b(\mathrm{Jac}(C)) \right\rangle. 
\end{align*}
The result of Corollary~\ref{cor:appl} extends the 
above results to an arbitrary $r \in \mathbb{Z}_{\ge 0}$. 

Another application is on semiorthogonal decompositions of 
Hilbert schemes of points on surfaces under blow-up. 
Let $S$ be a smooth projective surface and
$\widehat{S} \to S$ be a blow-up at a point. 
Then the G\"ottsche formula~\cite{Got}
for the Euler numbers of Hilbert schemes of points $\Hilb^n(S)$ 
in particular implies the blow-up formula 
\begin{align}\label{blow:Eu}
	\sum_{n\ge 0} e(\Hilb^n(\widehat{S})) q^n
	=\sum_{n\ge 0} e(\Hilb^n(S))q^n \cdot \prod_{d\ge 1}\frac{1}{(1-q^d)}. 
\end{align}
Note that if we define $p(j)$ to be the number of partitions of $j$, 
we have the formula 
	\begin{align*}
	\sum_{j\ge 0}p(j) q^j 
	=\prod_{d\ge 1}\frac{1}{(1-q^d)}. 
\end{align*}
On the other hand, Nakajima-Yoshioka~\cite{NakYos}
proved that $\Hilb^n(S)$ and $\Hilb^{n}(\widehat{S})$
are related by wall-crossing diagrams. 
One can show that each wall-crossing diagram fits into the framework of 
Quot formula in Theorem~\ref{intro:thm}, see~\cite[Theorem~4.1]{NakYos}, \cite[Theorem~4.1]{Kosekiup}. 
Based on this observation and using Theorem~\ref{intro:thm}, the following 
blow-up formula is obtained in~\cite{Kosekiup}: 
\begin{thm}\emph{(Koseki~\cite{Kosekiup})}\label{intro:thm2}
	There is a semiorthogonal decomposition of the form 
	\begin{align*}
	D^b(\Hilb^n(\widehat{S}))=\left \langle p(j)\mbox{-copies of }
	D^b(\Hilb^{n-j}(S)) : j=0, \ldots, n  \right\rangle. 	
		\end{align*}
		\end{thm}
The semiorthogonal decomposition in Theorem~\ref{intro:thm2} categorifies
the blow-up formula (\ref{blow:Eu}). 

\subsection{Notation and convention}\label{subsec:notation}
In this paper, all the varieties or (derived) stacks are
defined over $\mathbb{C}$. 
For an introduction to derived algebraic geometry, 
we refer to~\cite{MR3285853}. 
For a locally free sheaf $\eE$ on a stack $X$, 
we often regard it as a total space of its 
associated vector bundle, 
i.e. $\Spec \Sym(\eE^{\vee}) \to X$. 
For a derived Artin stack $\fM$, we denote by $t_0(\fM)$ 
its classical truncation. 
Explicitly if $\fM=[\Spec A^{\bullet}/G]$ for a commutative dg-algebra 
$A^{\bullet}$ with non-positive degrees and an algebraic group $G$ acting on $A^{\bullet}$, 
we have $t_0(\fM)=[\Spec \hH^0(A^{\bullet})/G]$, also see Remark~\ref{rmk:truncation}. 
For a complex of vector bundles $\eE^{\bullet}$ with differential 
$d_{\eE^{\bullet}}$ on $X$, 
we denote by $\mathrm{Sym}(\eE^{\bullet})$ the sheaf of dg-algebras on $X$, 
whose underlying graded sheaf is the super-symmetric product of 
$\eE^{\bullet}$, and the differential $d_{\mathrm{Sym}(\eE^{\bullet})}$ is uniquely determined by 
the condition that $d_{\mathrm{Sym}(\eE^{\bullet})}|_{\eE^{\bullet}}=d_{\eE^{\bullet}}$
and the Leibniz rule. 

For a derived stack $\fM$, the 
triangulated category $D^b(\fM)$ is defined to be 
the homotopy category of the $\infty$-category of 
quasi-coherent sheaves on $\fM$ with bounded coherent cohomologies. 
The tangent complex of $\fM$ is denoted by 
$\mathbb{T}_{\fM}$
 (see~\cite[Section~3.1]{MR3285853}), 
 and the cotangent complex $\mathbb{L}_{\fM}$ is defined to be its dual. 
A derived stack $\fM$ is called quasi-smooth if its cotangent 
complex $\mathbb{L}_{\fM}$ is perfect and
$\mathbb{L}_{\fM}|_{t_0(\fM)}$
has cohomological amplitude contained in $[-1, 1]$.
The rank of $\mathbb{L}_{\fM}|_{t_0(\fM)}$ is called the virtual 
dimension of $\fM$. 
For example if $\yY$ is a smooth (classical) Artin stack, 
$\eE \to \yY$ is a vector bundle with a section $s$, 
the derived fiber product 
$\yY \times_{0, \eE, s} \yY$ is quasi-smooth with virtual dimension 
$\dim \yY -\mathrm{rank}(\eE)$, which is called 
derived zero locus of $s$. 
When $\yY=\Spec A$ for a commutative $\mathbb{C}$-algebra and 
$\eE$ is determined by a projective $A$-module $M$, 
then the derived zero locus is $\Spec K(A, M, s)$, where 
$K(A, M, s)$ is the Koszul 
complex 
$\cdots \to\wedge M^{\vee} \stackrel{s}{\to} M^{\vee} \stackrel{s}{\to} A \to 0$, see~\cite[Last paragraph of Section~2.2]{MR3285853}.

	\subsection{Acknowledgements}
The author is grateful to Qingyuan Jiang
for the discussion on his conjecture
and several comments on the first draft of the paper, and 
Naoki Koseki on the discussion on the application of Theorem~\ref{intro:thm}
to Hilbert schemes of 
points on surfaces. 
The author also thanks to the referees for many valuable comments. 
The author is supported by World Premier International Research Center
Initiative (WPI initiative), MEXT, Japan, and Grant-in Aid for Scientific
Research grant (No.~19H01779) from MEXT, Japan.

\section{Proof of Theorem~\ref{intro:thm}}
\subsection{Derived structures of Quot schemes}\label{subsec:derived}
Let $X$ be a smooth quasi-projective variety over $\mathbb{C}$, 
$\mathscr{G}$ a coherent sheaf on it. 
Recall that the 
Quot scheme $\Quot_{X, d}(\mathscr{G})$ represents the functor 
\begin{align*}
	\qQ uot_{X, d}(\mathscr{G}) \colon 
	({Sch}/X)^{op} \to ({Set})
	\end{align*}
which sends $T \to X$ to the equivalence classes 
of $\mathscr{G}_T \twoheadrightarrow \pP$ where $\pP$ is a locally 
free sheaf on $T$ of rank $d$
and $\mathscr{G}_T$ is the pull-back of $\mathscr{G}$ to $T$. 

Let us take a right exact sequence
\begin{align}\label{loc:free}
	\mathscr{E}^{-1} \stackrel{\phi}{\to} \mathscr{E}^0 \to \mathscr{G} \to 0
	\end{align}
such that $\eE^i$ are locally free sheaves of finite rank on $X$. 
The surjection $\mathscr{E}^0 \twoheadrightarrow \mathscr{G}$
induces the closed immersion 
\begin{align}\label{quot:closed}
	\Quot_{X, d}(\mathscr{G}) \hookrightarrow 
	\Quot_{X, d}(\mathscr{E}^0). 
	\end{align}
Below we fix a vector space $V$ of dimension $d$, and 
denote by 
$\GL_X(V) \cneq \GL(V) \times X \to X$
the group scheme over $X$. 
We also set 
\begin{align*}
		\mathfrak{C}(\eE^0) \cneq 
	[\hH om(\eE^0, V \otimes \oO_X)/\GL_X(V)]. 
	\end{align*}
Here we have
identified the locally free sheaf $\hH om(\eE^0, V \otimes \oO_X)$ with 
the associated vector bundle over $X$, i.e. 
$\Spec \Sym(\eE^0 \otimes V^{\vee}) \to X$. 
\begin{lem}\label{lem:open}
There is an open immersion 
$\Quot_{X, d}(\mathscr{E}^0) \subset \mathfrak{C}(\eE^0)$. 
\end{lem}
\begin{proof}
For $T \to X$, the $T$-valued points 
of the stack $\mathfrak{C}(\eE^0)$
consist of 
$(\pP, s)$ where $\pP$ is a vector bundle on $T$ of rank $d$
and $s \colon \eE^0_T \to \pP$ is a morphism. 
Indeed giving a $X$-morphism 
$T \to \mathfrak{C}(\eE^0)$ is equivalent to 
giving a $\GL_T(V)$-torsor 
$\mathscr{F} \to T$ and a $\GL_T(V)$-equivariant 
morphism $\mathscr{F} \to \hH om(\eE^0_T, V \otimes \oO_T)$. 
The $\GL_T(V)$-torsor $\mathscr{F}$ corresponds to a vector bundle 
$\pP$ on $T$ such that $\mathscr{F}$ is isomorphic to the local framing of $\pP$, i.e. 
the set of sections of $\mathscr{F}$ over an \'{e}tale morphism $U \to T$
is the set of isomorphisms 
$\pP_{U} \stackrel{\cong}{\to} V \otimes \oO_U$. 
Then the $\GL_T(V)$-equivariant morphism 
$\mathscr{F} \to \hH om(\eE^0_{T}, V \otimes \oO_T)$
corresponds to a vector bundle morphism $\eE_T^0 \to \pP$ on $T$. 

From the definition of 
$\Quot_{X, d}(\eE^0)$, it is isomorphic to 
the open substack of $\mathfrak{E}(\eE^0)$ whose $T$-valued points 
correspond to $(\pP, s)$ such that $s$ is surjective. 
\end{proof}
We have the following vector bundle over 
$\mathfrak{C}(\eE^0)$ with a section $s$
\begin{align}\label{sec:hom}
		\xymatrix{
		\left[\left(\hH om(\eE^0, V \otimes \oO_X) \oplus \hH om(\eE^{-1}, V \otimes \oO_X)\right)/\GL_X(V)\right] 
	\ar[r] & \ar@/_20pt/[l]_s
\mathfrak{C}(\eE^0). 
}
	\end{align}
The section $s$ is induced by the $\GL_X(V)$-equivariant morphism 
\begin{align}\label{sec:hom2}
	s \colon \hH om(\eE^0, V \otimes \oO_X) \to 
	\hH om(\eE^0, V \otimes \oO_X) \oplus \hH om(\eE^{-1}, V \otimes \oO_X), \ 
	f \mapsto (f, f\circ \phi). 
	\end{align}
We denote by $\eE^{\bullet}$ the two term complex 
$(\eE^{-1} \stackrel{\phi}{\to} \eE^0)$ 
such that $\eE^0$ is of degree zero. 
Let $\mathfrak{C}(\eE^{\bullet})$ be the
derived zero locus of $s$.
The Koszul complex associated with $s$ is 
\begin{align*}
	\cdots \to \wedge^2 \eE^{-1} \otimes \mathrm{Sym}(\eE^0 \otimes V^{\vee})
	\stackrel{s}{\to} \eE^{-1} \otimes \mathrm{Sym}(\eE^0 \otimes V^{\vee})
	\stackrel{s}{\to} \mathrm{Sym}(\eE^0 \otimes V^{\vee}) \to 0
	\end{align*}
which coincides with $\mathrm{Sym}(\eE^{\bullet} \otimes V^{\vee})$, 
see Subsection~\ref{subsec:notation} for the dg-algebra 
structure on $\mathrm{Sym}(\eE^{\bullet} \otimes V^{\vee})$
over $X$.  
Therefore $\mathfrak{C}(\eE^{\bullet})$ is 
written as
\begin{align}\label{def:CE}
	\mathfrak{C}(\eE^{\bullet}) \cneq 
	\left[\Spec \Sym(\eE^{\bullet} \otimes V^{\vee})/\GL_X(V) \right]. 
	\end{align} 
Note that $\mathfrak{C}(\eE^{\bullet})$ is a derived closed substack 
of $\mathfrak{C}(\eE^0)$. 
We set 
\begin{align*}
	\mathbf{Quot}_{X, d}(\gG) \cneq 
	\Quot_{X, d}(\eE^0) \cap \mathfrak{C}(\eE^{\bullet}),  
	\end{align*}
in other word $\mathbf{Quot}_{X, d}(\gG)$ is the derived 
zero locus of $s$ restricted to the open substack 
$\Quot_{X, d}(\eE^0) \subset \mathfrak{C}(\eE^0)$. 
\begin{lem}\label{lem:vdimQUot}
	The derived stack 
	$\mathbf{Quot}_{X, d}(\gG)$
	has virtual dimension $\dim X+\delta d-d^2$,
with classical truncation 
$\Quot_{X, d}(\gG)$. 
\end{lem}
\begin{proof}
	The derived stack $\mathfrak{C}(\eE^{\bullet})$
	is a derived zero locus of $s$, 
	so it is quasi-smooth with virtual dimension 
	\begin{align*}
		\dim \mathfrak{C}(\eE^{0})-\rank(V \otimes \eE^{-1 \vee})
		&=\dim X +d \rank(\eE^0)-\dim \GL(V)-d \rank(\eE^{-1}) \\
		&=\dim X+\delta d-d^2. 
		\end{align*}	
		The derived stack 
	$\mathbf{Quot}_{X, d}(\gG)$
	is an open substack of $\mathfrak{C}(\eE^{\bullet})$, 
	so it also has virtual dimension $\dim X+\delta d-d^2$. 
	
	For a $X$-scheme $T \to X$, a $T$-valued point of the classical truncation 
	of $\mathbf{Quot}_{X, d}(\gG)$
	consists of a surjection $\eE^0_T \twoheadrightarrow \pP$
	such that the composition $\eE^{-1}_T \to \eE^0_T \twoheadrightarrow \pP$ is zero. 
	This is equivalent to giving a surjection 
	$\gG_T \twoheadrightarrow \pP$, i.e. 
	a $T$-valued point of $\Quot_{X, d}(\gG)$. 
	\end{proof}

By taking the dual of the sequence (\ref{loc:free}), 
we obtain the right exact sequence 
\begin{align*}
	\eE_0 \stackrel{\phi^{\vee}} \to \eE_1 \to \mathscr{H} \to 0. 
	\end{align*}
Here we have set $\eE_i:=(\eE^{-i})^{\vee}$, and $\mathscr{H}$ is defined 
to be the cokernel of $\phi^{\vee}$. 
We apply the above construction for the 
quotient $\eE_1 \twoheadrightarrow \mathscr{H}$.
By replacing 
$V$ with $V^{\vee}$ and noting $\GL_X(V)=\GL_X(V^{\vee})$, 
we have the closed immersion and an open immersion 
\begin{align}\notag
	\Quot_{X, d}(\mathscr{H}) \hookrightarrow \Quot_{X, d}(\eE_1)
	\subset \mathfrak{C}(\eE_1) \cneq [\hH om(\eE_1, V^{\vee} \otimes \oO_X)/\GL_X(V)]. 
	\end{align}
We also have the vector bundle with a section $s^{\vee}$
\begin{align}\label{sec:hom20}
	\xymatrix{
		\left[\left(\hH om(\eE_1, V^{\vee} \otimes \oO_X) \oplus 
		\hH om(\eE_0, V^{\vee} \otimes \oO_X)\right)/\GL_X(V)\right] 
		\ar[r] & \ar@/_20pt/[l]_{s^{\vee}}
		\mathfrak{C}(\eE_1). 
	}
\end{align}
The section $s^{\vee}$ is induced by the morphism 
\begin{align*}
	s^{\vee} \colon \hH om(\eE_1, V^{\vee} \otimes \oO_X) \to 
	\hH om(\eE_1, V^{\vee} \otimes \oO_X) \oplus 
	\hH om(\eE_0, V^{\vee} \otimes \oO_X), \ 
	f \mapsto (f, f\circ \phi^{\vee}). 
	\end{align*}
Similarly to (\ref{def:CE}), 
the derived zero locus of $s^{\vee}$ is written as 
\begin{align*}
	\mathfrak{C}(\eE_{\bullet}[1])
	\cneq \left[\Spec \Sym(\eE_{\bullet}[1] \otimes V)/\GL_X(V) \right].
	\end{align*}
Here $\eE_{\bullet}[1]$ is the complex 
$(\eE_0 \stackrel{\phi^{\vee}}{\to}
\eE_1)$ such that $\eE_1$ is of degree zero. 
We set 
\begin{align}\label{Quot:svee}
	\mathbf{Quot}_{X, d}(\mathscr{H})
	\cneq \Quot_{X, d}(\eE_1) \cap \mathfrak{C}(\eE_{\bullet}[1]). 
	\end{align}
The same proof of Lemma~\ref{lem:vdimQUot}
shows that 
$\mathbf{Quot}_{X, d}(\mathscr{H})$ 
has virtual dimension 
$\dim X-\delta d-d^2$, and its 
classical truncation is 
$\Quot_{X, d}(\mathscr{H})$. 

%\begin{rmk}\label{rmk:kerphi}
%	The constructions of $\mathfrak{C}(\eE^{\bullet})$, 
%	$\mathfrak{C}(\eE_{\bullet}[1])$ is not symmetric since 
%	$\Ker(\phi)=0$ and $\Ker(\phi^{\vee}) \neq 0$. However it does not 
%	affect the argument of Lemma~\ref{lem:vdimQUot}. 
%	\end{rmk}

%\begin{rmk}
%The derived 
%stacks 
%$\mathfrak{C}(\eE^{\bullet})$, $\mathfrak{C}(\eE_{\bullet}[1])$, 
%$\mathbf{Quot}_{X, d}(\gG)$, 
%$\mathbf{Quot}_{X, d}(\mathscr{H})$
%are intrinsic to $\gG$,
%and independent of a resolution 
%$\eE^{\bullet}$ up to equivalence. 
%\end{rmk}
\subsection{$(-1)$-shifted cotangent derived stacks}\label{subsec:cotan}
For a derived Artin stack $\fM$, its $(-1)$-shifted cotangent 
is defined by (see~\cite{Calaque})
\begin{align*}
	\Omega_{\fM}[-1] \cneq \Spec \Sym_{\oO_{\fM}}(\mathbb{T}_{\fM}[1]). 
	\end{align*}
Here $\mathbb{T}_{\fM}$ is the tangent complex of $\fM$. 

In the case that $\fM$ is a derived zero locus, the 
classical truncation of $\Omega_{\fM}[-1]$ has 
the following critical locus description. 
Let $\yY=[Y/G]$ for a smooth quasi-projective scheme $Y$
and $G$ is an affine algebraic group acting on $Y$. 
Let $\fF \to \yY$ be a vector bundle on it with a section $s$, 
which is identified with a $G$-equivariant
vector bundle $F \to Y$ together with a $G$-invariant 
 section $\widetilde{s}$ 
of $F \to Y$. 
Suppose that $\fM$ is a derived zero locus of $s$, 
that is 
$\fM=[\widetilde{M}/G]$ where $\widetilde{M}$ is the derived zero 
locus of $\widetilde{s}$.  
Let $w$ be the function
\begin{align}\label{func:w}
	w \colon \fF^{\vee} \to \mathbb{A}^1, \ 
	w(y, v)=\langle s(y), v \rangle
	\end{align}
for $y \in \yY$ and $v \in \fF^{\vee}|_{y}$, which 
is identified with a $G$-invariant 
function $\widetilde{w}$ on $F^{\vee}$. 
We set 
\begin{align*}
	\Crit(w) \cneq [\Crit(\widetilde{w})/G] \subset 
	\fF^{\vee}
	\end{align*}
which is a closed substack of $\fF^{\vee}$. 
Here $\mathrm{Crit}(\widetilde{w}) \subset F^{\vee}$ is the scheme theoretic 
critical locus of $\widetilde{w}$, 
defined by the ideal generated by the 
image of $d\widetilde{w} \colon T_{F^{\vee}} \to \oO_{F^{\vee}}$. 
(Alternatively $\mathrm{Crit}(w)$ is the closed substack of $\fF^{\vee}$ 
defined by the ideal 
generated by the image of $dw \colon \hH^0(\mathbb{T}_{\fF^{\vee}}) \to \oO_{\fF^{\vee}})$. 
\begin{lem}\label{lem:crit:M}
	Suppose that $\fM$ is the derived zero locus of a 
	section $s$ 
	of a vector bundle $\fF \to \yY$ 
	for a quotient stack $\yY=[Y/G]$
	as above. 
	Then the classical truncation $t_0(\Omega_{\fM}[-1])$ 
of $\Omega_{\fM}[-1])$ is isomorphic to 
$\mathrm{Crit}(w)$. 
\end{lem}
\begin{proof}
	We denote by $M \subset Y$ the classical truncation of $\widetilde{M}$, 
	that is the classical zero locus of 
	$\widetilde{s}$.  
	Note that $M \subset Y$ is a $G$-invariant closed subscheme, 
	and we have $\mM \cneq t_0(\fM)=[M/G]$, see Remark~\ref{rmk:truncation}. 
	The shifted tangent complex $\mathbb{T}_{\fM}[1]$ restricted to 
	$\mM$ is given by 
	\begin{align*}
		\mathbb{T}_{\fM}[1]|_{\mM}=(\mathfrak{g} \otimes \oO_{\mM}
		\to T_{\yY}|_{\mM} \stackrel{ds}{\to} \fF|_{\mM})
		\end{align*}
	where 
	$T_{\yY}=[T_Y/G]$ which is a vector bundle on $\yY$, 
	$\fF$ is located in degree zero. 
   In particular $\fM$ is quasi-smooth, see Subsection~\ref{subsec:notation} 
   for the 
   definition of quasi-smoothness. 
	Let us take a distinguished triangle in $D^b(\fM)$
	\begin{align*}
		\rR \to \mathbb{T}_{\fM}[1] \to \mathbb{T}_{\fM}[1]|_{\mM}. 
		\end{align*}
	Here we regarded the last term as an object in $D^b(\fM)$ by 
	the push-forward of the closed immersion $\mM \hookrightarrow \fM$.
	Then $\rR$ is concentrated in negative degrees, 
	$\mathbb{T}_{\fM}[1]$ and $\mathbb{T}_{\fM}[1]|_{\mM}$ are concentrated on 
	non-positive degrees. Therefore by taking the symmetric products and the zero-th cohomology, we 
	have 
	\begin{align*}
		\hH^0(\mathrm{Sym}_{\oO_{\fM}}(\mathbb{T}_{\fM}[1])) \stackrel{\cong}{\to}
		\hH^0(\mathrm{Sym}_{\oO_{\mM}}(\mathbb{T}_{\fM}|_{\mM}[1])). 
		\end{align*}
	We also have the distinguished triangle 
	\begin{align*}
		\mathfrak{g}\otimes \oO_{\mM}[1] \to (T_{\yY}|_{\mM} \stackrel{ds}{\to}
		\fF|_{\mM}) \to \mathbb{T}_{\fM}[1]
		\end{align*}
	where in the middle term $\fF|_{\mM}$ is located in degree zero. 
	Again by taking the symmetric products and the zero-th cohomology, 
	we obtain 
	\begin{align*}
		\hH^0(\mathrm{Sym}_{\oO_{\mM}}(T_{\yY}|_{\mM} 
		\stackrel{ds}{\to}\fF|_{\mM}))
		\stackrel{\cong}{\to}
			\hH^0(\mathrm{Sym}_{\oO_{\mM}}(\mathbb{T}_{\fM}|_{\mM}[1])). 
		\end{align*}
	Therefore the stack 
	$t_0(\Omega_{\fM}[-1])$ is isomorphic to the classical truncation of 
	\begin{align*}
		\Spec \mathrm{Sym}_{\oO_{\mM}}(T_{\yY}|_{\mM} 
		\stackrel{ds}{\to}\fF|_{\mM})=
		[\Spec \mathrm{Sym}_{\oO_M}(T_Y|_{M} \stackrel{d\widetilde{s}}{\to} F|_{M})/G]. 
		\end{align*}
	The classical truncation of the 
	derived scheme 
	$\Spec \mathrm{Sym}(T_Y|_{M} \stackrel{d\widetilde{s}}{\to} F|_{M})$
	is isomorphic to $\Crit(\widetilde{w})$
	(see~\cite[Proposition~2.8]{MR3607000}, ~\cite[Section~2.1.1]{TocatDT}), 
	therefore the lemma holds. 
	\end{proof}
\begin{rmk}\label{rmk:truncation}
	We use the fact taking the classical truncation $t_0(-)$ commutes 
	with taking the quotient stack. 
	Indeed let $\mathfrak{Y}$ be a derived scheme with a $G$-action, 
	and $Y=t_0(\mathfrak{Y})$. 
	The 
	quotient stack $[\mathfrak{Y}/G]$ is obtained as a 
	colimit of the simplicial derived scheme that is equal to 
	$G^{\times n} \times \mathfrak{Y}$ in degree $n$. 
	As $t_0(-)$ commutes with taking colimits, see~\cite[Paragraph after Definition~2.2.4.3]{HomII}, 
	we see that $t_0([\mathfrak{Y}/G])=[Y/G]$. 
		\end{rmk}

The above construction is summarized in the following diagram 
\begin{align}\label{dia:-1shift}
	\xymatrix{
\yY \ar@<-0.3ex>@{^{(}->}[r]^-{0} \diasquare& \fF \ar[d] \\
\fM \ar@<-0.3ex>@{^{(}->}[u] \ar@<-0.3ex>@{^{(}->}[r]
& \yY,  \ar@/_10pt/[u]_-{s} 
}	
\quad 
\xymatrix{t_0(\Omega_{\fM}[-1]) \ar[r]^-{\cong} \ar[d] &
\Crit(w) \ar@<-0.3ex>@{^{(}->}[r] \ar[d] & \fF^{\vee} \ar[d] \ar[rd]^-{w} & \\
\fM \ar@{=}[r]& \fM \ar@<-0.3ex>@{^{(}->}[r] & \yY & \mathbb{A}^1. 
}
	\end{align}
Here the left square is a derived Cartesian.

We return to the setting of the previous subsections. 
Let $V$ be a $d$-dimensional vector space. 
We set $Y(d)$ and $\yY(d)$ to be 
\begin{align}\label{def:Yd}
	&Y(d) \cneq \hH om(\eE^0, V \otimes \oO_X) \oplus 
	\hH om(V \otimes \oO_X, \eE^{-1}), \\ 
\notag	&\yY(d) \cneq [Y(d)/\GL_X(V)]. 
	\end{align}
Again we have regarded $Y(d)$ as the total space 
of a vector bundle over $X$. 
For $T \to X$, the $T$-valued points 
of the stack $\yY(d)$ consist of tuples
\begin{align}\label{tuple}
	(\pP, \alpha, \beta), \ \alpha \colon \eE^0_T \to \pP, \ \beta \colon 
	\pP \to \eE_T^{-1}
	\end{align}
where $\pP$ is a locally free sheaf on $T$ of rank $d$. 
Note that the projection 
\begin{align*}
	\yY(d) \to [\hH om(\eE^0, V \otimes \oO_X)/\GL_X(V)]=\mathfrak{C}(\eE^0)
	\end{align*}
identifies $\yY(d)$ 
with the dual vector bundle of (\ref{sec:hom}). 
We define the super-potential 
\begin{align}\label{funct:w}
	w \colon \yY(d) \to \mathbb{A}^1, \ 
	(\pP, \alpha, \beta) \mapsto \langle s(\alpha), \beta \rangle
	=\mathrm{Tr}(\alpha \circ \phi_T \circ \beta). 
\end{align}
Here over the $T$-valued points, 
the 
last expression is given by taking the trace of the composition 
\begin{align*}
	\alpha \circ \phi_T \circ \beta \colon 
	\pP \stackrel{\beta}{\to} \eE_T^{-1} \stackrel{\phi_T}{\to} \eE_T^0 \stackrel{\alpha}{\to}
	\pP. 
\end{align*}
From the diagram (\ref{dia:-1shift}), 
Lemma~\ref{lem:crit:M} 
(applied for $\yY=\mathfrak{C}(\eE^0)$, 
$\fF$ is the vector bundle (\ref{sec:hom})
so that 
$\fF^{\vee}=\yY(d)$, the section $s$ is (\ref{sec:hom2}))
implies that 
we have the isomorphism 
\begin{align}\label{crit:1}
	\Crit(w) \stackrel{\cong}{\to} t_0(\Omega_{\mathfrak{C}(\eE^{\bullet})}[-1]). 
	\end{align}
\begin{rmk}\label{rmk:quiver}
	Let $a=\rank(\eE^0)$ and $b=\rank(\eE^{-1})$, 
	and denote by $Q_{a, b}$ the quiver with two vertices $\{0, 1\}$, 
	the $a$-arrows from $0$ to $1$ and $b$-arrows from $1$ to $0$. 
	We denote by 
	\begin{align*}
		\rR_{Q_{a, b}}(d) \cneq 
		[(V^{\oplus a} \oplus V^{\vee \oplus b})/\GL(V)], 
		\end{align*}
	the moduli stack of representations of 
	$Q_{a, b}$ with dimension vector $(1, d)$ for $d=\dim V$. 
	If $X$ is a point, then $\yY(d)$ is isomorphic to $\rR_{Q_{a, b}}(d)$. 
	In general there is a projection 
	$h \colon \yY(d) \to X$
	whose fiber is isomorphic to 
	$\rR_{Q_{a, b}}(d)$. 
	Moreover $\yY(d) \cong \rR_{Q_{a, b}}(d) \times X$ if 
	$\eE^0$ and $\eE^{-1}$ are free $\oO_X$-modules. 
	\end{rmk}

We have the isomorphism 
\begin{align}\label{isom:dual}
	\yY(d) \stackrel{\cong}{\to}
	\left[\left(\hH om(\eE_1, V^{\vee} \otimes \oO_X) \oplus \hH om(V^{\vee} \otimes \oO_X, \eE_0) \right) /\GL_X(V) \right]
\end{align}
by the correspondence over $T$-valued points
\begin{align*}
	(\pP, \alpha, \beta) 
	\mapsto (\pP^{\vee}, \beta^{\vee}, \alpha^{\vee}), \ \alpha^{\vee}
	 \colon \pP^{\vee} \to (\eE_0)_T, \ \beta^{\vee} \colon 
	 (\eE_1)_T \to \pP^{\vee}. 
	\end{align*}
Under the isomorphism (\ref{isom:dual}), the 
projection 
\begin{align*}
	\yY(d) \to [\hH om(\eE_1, V^{\vee} \otimes \oO_X)/\GL_X(V)]=\mathfrak{C}(\eE_1)
	\end{align*}
identifies the 
 stack $\yY(d)$ with the 
dual vector bundle of (\ref{sec:hom20}). 
Moreover under the isomorphism (\ref{isom:dual}), 
the super-potential (\ref{funct:w}) is also identified with 
\begin{align*}
	w(\alpha, \beta)=\langle s^{\vee}(\beta^{\vee}), \alpha^{\vee}\rangle
	=\mathrm{Tr}(\beta^{\vee}\circ \phi^{\vee} \circ \alpha^{\vee}),
	\end{align*}
where over the $T$-valued points, 
the last expression is the trace for the composition
\begin{align*}
	\beta^{\vee}\circ \phi^{\vee} \circ \alpha^{\vee} \colon 
		\pP^{\vee} \stackrel{\alpha^{\vee}}{\to} (\eE_0)_T
		 \stackrel{\phi^{\vee}}{\to} (\eE_1)_T \stackrel{\beta^{\vee}}{\to}
	\pP^{\vee}.
	\end{align*}
Therefore again by Lemma~\ref{lem:crit:M}, 
we also have the isomorphism 
\begin{align}\label{equiv:crit}
\Crit(w) \stackrel{\cong}{\to}  t_0(\Omega_{\mathfrak{C}(\eE_{\bullet}[1])}[-1]). 
	\end{align}

Let $\chi_0$ be the determinant character of $\GL(V)$
\begin{align}\label{det:chi0}
	\chi_0 \colon \GL(V) \to \C, \ g \mapsto \det g, 
\end{align}
which naturally determines a line bundle on $\yY(d)$, denoted 
by the same symbol $\chi_0$. 
\begin{lem}\label{lem:GIT}
The GIT semistable locus
\begin{align*}
	\yY(d)^{\chi_0 \sss}\subset \yY(d), \ 
	\yY(d)^{\chi_0^{-1}\sss} \subset \yY(d)
\end{align*}
consists of $(\pP, \alpha, \beta)$ in (\ref{tuple}) 
such that $\alpha$ is 
surjective, $\beta^{\vee}$ is surjective, respectively. 
\end{lem}
\begin{proof}
	We only prove the case of $\yY(d)^{\chi_0 \sss}$. 
	By the Hilbert-Mumford criterion 
	in terms of the $\Theta$-stack 
	$\Theta \cneq [\mathbb{A}^1/\mathbb{C}^{\ast}]$
	(see~\cite{Halpinstab}),  
the semistable locus $\yY(d)^{\chi_0 \sss}$
consists of $p \in \yY(d)$ such that 
for any $g \colon \Theta \to \yY(d)$ with 
$g(1)=p$, we have $\wt(g(0)^{\ast}\chi_0) \ge 0$. 
Since $\Theta \to \Spec \mathbb{C}$ is the 
good moduli space for $\Theta$, see~\cite[Example~8.2]{MR3237451}, 
any map $g\colon \Theta \to \yY(d)$
composed with the projection $\yY(d) \to X$ factors through 
$\Theta \to \Spec \mathbb{C}$
by the universal property of the good moduli space, 
see~\cite[Theorem~6.6]{MR3237451}. 
Therefore any map $g \colon \Theta \to \yY(d)$ 
is contained in a fiber of
$\yY(d) \to X$. 
 Moreover $\alpha$ is surjective if and only if $\alpha|_{x}$ is 
surjective for any $x \in X$. Therefore we may assume that $X$ is a point.
In this case, the lemma follows from~\cite[Lemma~5.1.9]{TocatDT}.  
	\end{proof}
Let us take the GIT quotient 
\begin{align*}
	\yY(d) \to Y(d)\ssslash \GL_X(V)\cneq 
	\Spec (h_{\ast}\oO_{Y(d)})^{\GL_X(V)}
	\end{align*}
where $h \colon Y(d) \to X$ is the projection. 
The above morphism is 
a good moduli space morphism for $\yY(d)$ 
in the sense of~\cite{MR3237451}, see~\cite[Theorem~13.2]{MR3237451}.  
We have the commutative diagram 
\begin{align}\label{dia:Y}
	\xymatrix{
\yY(d)^{\chi_0 \sss} \ar[r] \ar[rd]_-{w^+} & Y(d)\ssslash \GL_X(V) \ar[d]_-{w}& \yY(d)^{\chi_0^{-1}\sss} \ar[l] \ar[ld]^-{w^-} \\
& \mathbb{A}^1. &	
}
	\end{align}
\begin{lem}\label{lem:restrict}
The equivalences (\ref{crit:1}), (\ref{equiv:crit}) restrict to isomorphisms 
\begin{align}\label{Crit:Omega}
	\Crit(w^+) \stackrel{\cong}{\to}
	t_0(\Omega_{\mathbf{Quot}_{X, d}(\gG)}[-1]), \ 
	\Crit(w^-) \stackrel{\cong}{\to}
t_0(\Omega_{\mathbf{Quot}_{X, d}(\mathscr{H})}[-1]).	
	\end{align}
\end{lem}
\begin{proof}
	We only prove the first isomorphism. 
	Lemma~\ref{lem:GIT} implies that the following diagram is 
	Cartesian 
	\begin{align}\label{cart:crit0}
		\xymatrix{
		\yY(d)^{\chi_0 \sss} \ar@<-0.3ex>@{^{(}->}[r] \ar[d] \diasquare 
		&
		\yY(d) \ar[d] \\
		\mathrm{Quot}_{X, d}(\eE^0) \ar@<-0.3ex>@{^{(}->}[r]
		& \mathfrak{C}(\eE^0) 
	}
		\end{align}
	where each horizontal arrow is an open immersion. 
	Note that $\mathrm{Crit}(w) \cap \yY(d)^{\chi_0 \sss}=\mathrm{Crit}(w^{+})$
	as $\yY(d)^{\chi_0\sss} \subset \yY(d)$ is an open immersion. 
	Therefore we obtain the 
	Cartesian square 
				\begin{align}\label{cart:critw}
			\xymatrix{
				\Crit(w^{+}) \ar@<-0.3ex>@{^{(}->}[r] \ar[d] \diasquare 
				&
				\Crit(w) \ar[d] \\
				\mathrm{Quot}_{X, d}(\eE^0) \ar@<-0.3ex>@{^{(}->}[r]
				& \mathfrak{C}(\eE^0).  
			}
		\end{align}
	We also have the following Cartesian diagrams 
	from the definition of $(-1)$-shifted cotangents and
	$\mathbf{Quot}_{X, d}(\gG)$
	\begin{align}\label{cart:critw2}
			\xymatrix{
			t_0(\Omega_{\mathbf{Quot}_{X, d}(\gG)}[-1]) \ar@<-0.3ex>@{^{(}->}[r] \ar[d] \diasquare 
			&
			t_0(\Omega_{\mathfrak{C}(\eE^{\bullet})}[-1]) \ar[d] \\
			\mathbf{Quot}_{X, d}(\gG) \ar@<-0.3ex>@{^{(}->}[r]
			& \mathfrak{C}(\eE^{\bullet}), 
		} \quad 
		\xymatrix{
	  \mathbf{Quot}_{X, d}(\gG) \ar@<-0.3ex>@{^{(}->}[r]  \ar@<-0.3ex>@{^{(}->}[d] \diasquare 
		& \mathfrak{C}(\eE^{\bullet})  \ar@<-0.3ex>@{^{(}->}[d] \\
	\mathrm{Quot}_{X, d}(\eE^0) \ar@<-0.3ex>@{^{(}->}[r]
		& \mathfrak{C}(\eE^{0}). 
	}
		\end{align} 
	The lemma follows from the Cartesian squares (\ref{cart:critw}), (\ref{cart:critw2})
	 together 
	with the isomorphism (\ref{crit:1}). 
	\end{proof}
When $X$ is a point, the top row in (\ref{dia:Y}) is 
a Grassmannian flip considered in~\cite[(4.5)]{Toconi}. 
In general, it is a family of Grassmannian 
flips parametrized by $X$. 
\begin{rmk}\label{rmk:dcrit}
	The diagram 
	\begin{align}\notag
		\xymatrix{
		t_0(\Omega_{\mathbf{Quot}_{X, d}(\gG)}[-1]) \ar[rd] & & 
		t_0(\Omega_{\mathbf{Quot}_{X, d}(\mathscr{H})}[-1])
\ar[ld] \\
& Y(d)\ssslash \GL_X(V) &	
}	\end{align}
is a d-critical flip in~\cite{Toddbir}. 
If $\Quot_{X, d}(\gG)$ and $\Quot_{X, d}(\mathscr{H})$ are smooth of expected 
dimensions, 
then the above diagram is identified with 
	\begin{align*}
	\xymatrix{
	\Quot_{X, d}(\gG) \ar[rd] & & \Quot_{X, d}(\mathscr{H})		\ar[ld] \\
		& Y(d)\ssslash \GL_X(V). &	
}	\end{align*}
In general, the above diagram is not necessary a d-critical flip 
since the
relative Quot schemes are not  
necessary written as critical loci. 
	\end{rmk}
\subsection{Koszul duality}
We apply Koszul duality equivalences to relate derived categories of 
relative Quot schemes with triangulated categories of $\C$-equivariant 
factorizations. 
Below we use the convention 
in~\cite[Section~2.2, 2.3]{KoTo}. 

Let $\widetilde{\GL}(V)$ be defined by 
\begin{align*}
	\widetilde{\GL}(V) \cneq \GL(V) \times_{\PGL(V)} \times \GL(V). 
	\end{align*}
There is a natural exact sequence 
\begin{align}\label{char:tau}
	1 \to \GL(V) \stackrel{\Delta}{\to} \widetilde{\GL}(V) 
	\stackrel{\tau}{\to} \C \to 1
	\end{align}
where $\Delta$ is the diagonal embedding, 
and $\tau$ is the character defined by $\tau(g_1, g_2)=g_1 g_2^{-1}$. 
The above exact sequence splits non-canonically. 
Indeed for each $k \in \mathbb{Z}$, 
$t \mapsto (t^k, t^{k-1})$ 
gives a splitting of $\tau$. 
So for each $k \in \mathbb{Z}$, 
there is an isomorphism 
\begin{align}\label{tauk}
	\iota_k \colon 
	\GL(V) \times \C \stackrel{\cong}{\to} \widetilde{\GL}(V), 
	\ (g, t) \mapsto (t^k g, t^{k-1} g). 
	\end{align}
Once we fix a splitting as above, 
giving a $\widetilde{\GL}(V)$-action is equivalent to 
giving a $\GL(V)$-action together with an auxiliary $\C$-action 
which commutes with the above $\GL(V)$-action. 
The $\GL_X(V)$-action on $Y(d)$ over $X$ naturally extends to 
an action of $\widetilde{\GL}_X(V)\cneq \widetilde{\GL}(V) \times X$ over $X$. 
Indeed $\GL_X(V) \times_X \GL_X(V)$ naturally 
acts on $Y(d)$, 
where the first factor of $\GL_X(V) \times_X \GL_X(V)$ acts on 
$\hH om(\eE^0, V \otimes \oO_X)$ and the second factor 
acts on $\hH om(V \otimes \oO_X, \eE^{-1})$, 
and the $\widetilde{\GL}_X(V)$-action is 
given by its restriction. For $k=0$ in (\ref{tauk}), 
the auxiliary $\C$-action 
is given by the weight one action on the second factor of $Y(d)$, 
for $k=1$ it is the weight one action on the first factor of $Y(d)$. 

The triangulated category of 
$\C$-equivariant 
factorizations
\begin{align}\label{MF:cat}
	\MF^{\C}(\yY(d), w)
	\end{align}
is defined to be the category whose objects consist of 
\begin{align}\label{fact:P}
	P_0 \stackrel{f}{\to} P_1 \stackrel{g}{\to} P_0\langle\tau\rangle
	\end{align}
where $P_0$, $P_1$ are $\widetilde{\GL}_X(V)$-equivariant 
coherent sheaves on $Y(d)$, 
$f$, $g$ are $\widetilde{\GL}_X(V)$-equivariant morphisms 
such that $f \circ g=\cdot w$, 
$g \circ f =\cdot w$. 
Here $\langle \tau \rangle$ means the twist by the $\widetilde{\GL}(V)$-character 
$\tau$. The category (\ref{MF:cat}) is defined to be the localization
of the homotopy category of the factorizations (\ref{fact:P})
by its subcategory of acyclic factorizations (see~\cite{MR3366002}). 
The categories $\mathrm{MF}^{\mathbb{C}^{\ast}}(\yY(d)^{\chi_0^{\pm 1} \sss}, w^{\pm})$
are also defined in a similar way. 

We now state 
the Koszul duality equivalence in~\cite[Proposition~4.8]{MR3631231}
(also see~\cite{MR3071664, MR2982435, TocatDT})
in the setting of the diagram (\ref{dia:-1shift}):
\begin{thm}\emph{(\cite{MR3631231, MR3071664, MR2982435, TocatDT})}\label{thm:Kdual}
Let $\yY=[Y/G]$ for a smooth quasi-projective scheme $Y$
and $G$ is an affine algebraic group acting on $Y$. 
Let $\fF \to \yY$ be a vector bundle on it with a section $s$, 
and $\fM$ the derived zero locus of $s$. 
Then there is 
an equivalence 
\begin{align*}
	D^b(\fM) \stackrel{\sim}{\to} \MF^{\C}(\fF^{\vee}, w)
	\end{align*}
where $\C$ acts on fibers of $\fF^{\vee} \to Y$ with weight one, 
and $w$ is the function (\ref{func:w}). 
\end{thm}
By applying Theorem~\ref{thm:Kdual}, we obtain the following: 
\begin{prop}\label{prop:KdQuot}
We have the equivalences 
\begin{align}\label{equiv:quot}
&D^b(\mathbf{Quot}_{X, d}(\gG)) \stackrel{\sim}{\to}
\MF^{\C}(\yY(d)^{\chi_0 \sss}, w^+), \\ 
	&\notag D^b(\mathbf{Quot}_{X, d}(\mathscr{H})) \stackrel{\sim}{\to}
	\MF^{\C}(\yY(d)^{\chi_0^{-1} \sss}, w^-). 
	\end{align}
\end{prop}
\begin{proof}
	We apply Theorem~\ref{thm:Kdual} for 
	$\mathcal{Y}=\mathrm{Quot}_{X, d}(\eE^0)$ and 
	the vector bundle $\fF \to \yY$ with section $s$ 
	given by the 
	pull-back of (\ref{sec:hom}) by the open immersion $\mathrm{Quot}_{X, d}(\eE^0) \subset \mathfrak{C}(\eE^0)$. 
	Then from the Cartesian square (\ref{cart:crit0}), 
	we obtain the first equivalence in (\ref{equiv:quot})
	by Theorem~\ref{thm:Kdual}. 
	Here we have used the choice of splitting (\ref{tauk}) 
	for $k=0$ in order to specify the auxiliary $\mathbb{C}^{\ast}$-action. 
		The second equivalence in (\ref{equiv:quot})
	 is similarly proved using another splitting (\ref{tauk}) for $k=1$. 
\end{proof}
%\begin{rmk}
%	The statement of 
%	Koszul duality equivalence in~\cite[Proposition~4.8]{MR3631231}
%	is slightly different from the one in other articles. 
%	For example in~\cite[Theorem~2.3.3]{TocatDT}, 
%it is stated for fiberwise weight two $\C$-action rather than the weight one action. 
%This is because of the difference of 
%the definition of $\C$-equivariant factorizations, 
%where in the latter (\ref{fact:P}) is replaced by 
%$(P, d)$ for $d \colon P \to P\langle \tau\rangle$ with $d^2=w$. 
%Both definitions are indeed equivalent by assigning (\ref{fact:P}) 
%with $P=P^0 \oplus P^1$ where $\mu_2 \subset \C$ 
%acts on $P^i$ with weight $i$. 
%	\end{rmk}
\subsection{Window subcategories}
We fix a basis of $V$ and a Borel subgroup $B \subset \GL(V)$ 
to be consisting of upper triangular matrices, 
and set roots of $B$ to be negative roots. 
Let $M=\mathbb{Z}^d$ be the character lattice for $\GL(V)$, 
and $M^+_{\mathbb{R}} \subset M_{\mathbb{R}}$ the dominant chamber. 
By the above choice of negative roots, we have 
\begin{align*}
	M_{\mathbb{R}}^+=\{(x_1, x_2, \ldots, x_d) \in \mathbb{R}^d : 
	x_1 \le x_2 \le \cdots \le x_d\}. 
\end{align*} 
We set $M^+ \cneq M_{\mathbb{R}}^+ \cap M$. 
For $c \in \mathbb{Z}$, we set 
\begin{align}\label{Bcd}
	\mathbb{B}_{c}(d) \cneq \{(x_1, x_2, \ldots, x_d) \in M^+ : 
	0 \le x_i \le c-d\}. 
\end{align}
Here $\mathbb{B}_c(d)=\emptyset$ if $c<d$.
\begin{rmk}
For $\chi \in \mathbb{B}_c(d)$, we have the associated Young diagram 
whose number of boxes at the $i$-th row is $x_{d-i+1}$.
The above assignment 
identifies $\mathbb{B}_c(d)$ with the set of Young diagrams 
with height less than or equal to $d$, 
width less than or equal to $c-d$. 
For example, the following 
picture illustrates the case of 
$(2, 5, 5, 8) \in \mathbb{B}_{c}(d)$ for $d=4$
and $c\ge 12$:
\begin{figure}[H]
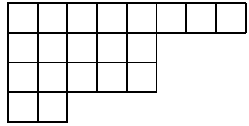
\caption{$(2, 5, 5, 8) \in \mathbb{B}_{c}(d), d=4, c\ge 12$}
	\begin{align*}
		\yng(8,5,5,2)
	\end{align*}
\end{figure}
\end{rmk}
 
By fixing a splitting (\ref{tauk}), 
we define the triangulated subcategory 
\begin{align}\label{subcat:W}
	\mathbb{W}_c(d) \subset \MF^{\C}(\yY(d), w)
	\end{align}
to be split generated by  
factorizations whose entries are 
of the form 
$V(\chi) \otimes_{\oO_X} \pP\langle \tau^i \rangle$
for $\chi \in \mathbb{B}_c(d)$, $i \in \mathbb{Z}$
and $\pP \in D^b(X)$. 
Here $V(\chi)$ is the irreducible $\GL(V)$-representation 
with highest weight $\chi$ 
(i.e. the Schur power of $V$ 
associated with the Young diagram determined by $\chi$), 
and
 $\tau \colon \widetilde{\GL}(V) \to \C$ is the character in (\ref{char:tau}). 
 Note that the subcategory (\ref{subcat:W}) does not depend on a choice of a splitting (\ref{tauk}), 
 since a different splitting only affects on $V(\chi) \otimes_{\oO_X}\pP\langle \tau^i \rangle$
 by a power of $\tau$. 
We also set
\begin{align}\label{abdelta}
	a \cneq \rank(\eE^0), \ b \cneq \rank(\eE^{-1}), \ 
	\delta=a-b. 
	\end{align}
\begin{lem}\label{lem:window}
	The following compositions are equivalences
	\begin{align*}
		&\mathbb{W}_a(d) \subset
		\MF^{\C}(\yY(d), w) \twoheadrightarrow 
		\MF^{\C}(\yY(d)^{\chi_0 \sss}, w^+), \\
		&\mathbb{W}_b(d) \subset
		\MF^{\C}(\yY(d), w) \twoheadrightarrow 
		\MF^{\C}(\yY(d)^{\chi_0^{-1} \sss}, w^-). 
		\end{align*}	
	\end{lem}
\begin{proof}
	We only prove the first equivalence. 
	The lemma is proved in~\cite[Proposition~4.3]{Toconi} when $X$ is a point and 
	there is no super-potential and an auxiliary $\C$-action, i.e. $D^b(\yY(d))$ instead 
	of $\MF^{\C}(\yY(d), w)$. 
	Namely let $\mathbb{W}_a'(d) \subset D^b(\yY(d))$ be the triangulated subcategory 
	generated by $V(\chi) \otimes_{\oO_X}\pP$ for $\chi \in \mathbb{B}_a(d)$
	and $\pP \in D^b(X)$. 
	If $X$ is a point, then the composition functor 
	\begin{align}\label{equiv:Wa}
		\mathbb{W}_a'(d) \subset D^b(\yY(d)) \twoheadrightarrow D^b(\yY^{\chi_0 \sss}(d))
		\end{align}
	is an equivalence by~\cite[Proposition~4.3]{Toconi}. 
	If $\eE^i$ are free $\oO_X$-modules so that 
	$\yY(d) \cong \rR_{Q_{a, b}}(d) \times X$
	(see Remark~\ref{rmk:quiver} for the notation $\rR_{Q_{a, b}}(d)$), then 
	(\ref{equiv:Wa}) is an equivalence by 
	taking the $\boxtimes$-product of the
equivalence (\ref{equiv:Wa}) in the case of $\yY(d)=\rR_{Q_{a, b}}(d)$
with $D^b(X)$. 
	For a general $X$, let 
	us take the factorization 
	\begin{align*}
		\yY(d) \stackrel{\pi}{\to} [X/\GL_X(V)] \to X
		\end{align*}
	where $\GL_X(V)$ acts on $X$ trivially (so $[X/\GL_X(V)]=X \times B\GL(V)$), and 
	$\pi$ is a natural morphism 
	induced by the projection 
	$\hH om(\eE_1, V^{\vee} \otimes \oO_X) \oplus \hH om(V^{\vee} \otimes \oO_X, \eE_0) \to X$, 	
	which is an affine space bundle.
	For $\chi, \chi' \in \mathbb{B}_a(d)$ and 
	$\pP, \pP' \in D^b(X)$, 
	we have the natural morphism in $D_{\rm{qcoh}}([X/\GL_X(V)])$
	\begin{align}\label{mor:GLX}
		&\hH om_{[X/\GL_X(V)]}(V(\chi) \otimes \pP, V(\chi') \otimes \pP' \otimes \pi_{\ast}\oO_{\yY(d)}) \\ &\notag \to 
		\hH om_{[X/\GL_X(V)]}(V(\chi) \otimes \pP, V(\chi') \otimes \pP' \otimes \pi_{\ast}\oO_{\yY(d)^{\chi_0 \sss}}).
				\end{align}
			For a Zariski open subset $U \subset X$, we write 
			$\yY(d)_{U} :=\pi^{-1}([U/\GL_U(V)])$
			and $\pi_U \colon \yY(d)_{U} \to [U/\GL_U(V)]$ the 
			restriction of $\pi$ to $\yY(d)_U$. 
			Then we have 
			\begin{align*}
				&\dR \Gamma([U/\GL_U(V)], \hH om_{[X/\GL_X(V)]}(V(\chi) \otimes \pP, V(\chi') \otimes \pP' \otimes \pi_{\ast}\oO_{\yY(d)})) \\
				&=\RHom_{\yY(d)_U}(\pi_U^{\ast}(V(\chi) \otimes \pP|_{U}), 
				\pi_U^{\ast}(V(\chi') \otimes \pP'|_{U})), \\
				&\dR \Gamma([U/\GL_U(V)], \hH om_{[X/\GL_X(V)]}(V(\chi) \otimes \pP, V(\chi') \otimes \pP' \otimes \pi_{\ast}\oO_{\yY(d)^{\chi_0 \sss}})) \\
				&=\RHom_{\yY(d)^{\chi_0 \sss}_U}(\pi_U^{\ast}(V(\chi) \otimes \pP|_{U})|_{\yY(d)^{\chi_0 \sss}_U}), 
				\pi_U^{\ast}(V(\chi') \otimes \pP'|_{U})|_{\yY(d)^{\chi_0 \sss}_U}). 				
				\end{align*}
			Therefore the morphism (\ref{mor:GLX}) is
			 an isomorphism Zariski locally on $X$
			(by the equivalence (\ref{equiv:Wa}) when $\eE^i$ are free),
			hence (\ref{mor:GLX}) is an isomorphism. By taking 
			$\dR \Gamma([X/\GL_X(V)], -)$ of the isomorphism (\ref{mor:GLX}), 
			we see that the functor (\ref{equiv:Wa}) is fully-faithful. 
			For the essential surjectivity of (\ref{equiv:Wa}), 
			we modify the argument of~\cite[Proposition~4.3]{Toconi}
			by replacing Kapranov exceptional collections on Grassmannians 
			with relative exceptional collections of Grassmannian bundles in~\cite[Theorem~3.70]{JiangQuot}. 
			
		The above
	argument applies verbatim with an auxiliary $\C$-action. 
	Namely let $\C$ acts on $\yY(d)$ with weight one on the second factor, 
	and $\mathbb{W}_a''(d) \subset D^b([\yY(d)/\C])$ the triangulated subcategory generated 
	by $V(\chi) \otimes_{\oO_X}\pP\langle \tau^i \rangle$
	for $\chi \in \mathbb{B}_a(d)$, $\pP \in D^b(X)$ and $i \in \mathbb{Z}$ where $\tau$ is the 
	weight one character for the auxiliary $\C$-action. 
	Then the composition functor 
	\begin{align}\notag
		\mathbb{W}_a''(d) \subset D^b([\yY(d)/\C]) \twoheadrightarrow D^b([\yY(d)^{\chi_0 \sss}/\C])
	\end{align}
is an equivalence. 
	Then
	the lemma holds by applying the super-potential $w$ to the above 
	equivalence
		 	(e.g.~applying~\cite[Proposition~2.1]{Tudor1.5} for $\xX=[\yY(d)/\mathbb{C}^{\ast}]$, 
		 	$I=\{1\}$, $\aA_1=\mathbb{W}_a''(d)$). 	
	\end{proof}

\subsection{Categorified Hall products}
For a one parameter subgroup $\lambda \colon \C \to \GL(V)$, 
let $V^{\lambda \ge 0} \subset V$ be the subspace of non-negative $\lambda$-weights, 
and $V^{\lambda=0} \subset V$ the $\lambda$-fixed subspace. 
We have the Levi and parabolic subgroups 
\begin{align*}
	\GL(V)^{\lambda=0} \subset \GL(V)^{\lambda \ge 0} \subset \GL(V)
	\end{align*}
where $\GL(V)^{\lambda=0}$ is the centralizer of $\lambda$
and $\GL(V)^{\lambda \ge 0}$ is the subgroup of $g \in \GL(V)$
such that there is a limit of $\lambda(t) g \lambda(t)^{-1} \in \GL(V)$ for $t \to 0$. 
We set 
\begin{align*}
&\yY(d)^{\lambda \ge 0} \cneq \left[\left(\hH om(\eE^0, V^{\lambda \ge 0}\otimes \oO_X) \oplus 
\hH om(V^{\lambda \le 0} \otimes \oO_X, \eE^{-1})\right)/\GL_X(V)^{\lambda \ge 0}    \right] \\
&\yY(d)^{\lambda=0} \cneq \left[\left(\hH om(\eE^0, V^{\lambda=0}\otimes \oO_X) \oplus 
\hH om(V^{\lambda = 0} \otimes \oO_X, \eE^{-1})\right)/\GL_X(V)^{\lambda=0}    \right]. 
	\end{align*}
Here the right hand sides make sense since the $\GL(V)$-action on $V$
restricts to the $\GL(V)^{\lambda \ge 0}$-action on $V^{\lambda \ge 0}$. 
We have the following diagram 
\begin{align}\label{dia:loci}
	\xymatrix{
\yY(d)^{\lambda \ge 0} \ar[d]_-{q_{\lambda}} \ar[r]^-{p_{\lambda}}
\ar[rd]^-{w^{\lambda \ge 0}} & \yY(d) \ar[d]^-{w} \\
\yY(d)^{\lambda =0} \ar[r]_{w^{\lambda=0}} & \mathbb{A}^1. 
}
	\end{align}
Here $p_{\lambda}$ is induced by 
the natural inclusion $V^{\lambda \ge 0} \subset V$
and surjection $V \twoheadrightarrow V^{\lambda \le 0}$, 
and $q_{\lambda}$ is given by 
taking the $t\to 0$ limit of the $\lambda$-action. 
\begin{rmk}\label{rmk:loci}
	The morphisms $p_{\lambda}$, $q_{\lambda}$ are morphisms of algebraic 
	stacks. Indeed the diagram 
	$\yY(d)^{\lambda=0} \leftarrow \yY(d)^{\lambda \ge 0} \to \yY(d)$
	is identified with some components of the diagram 
\begin{align*}
	\mathrm{Map}(B\mathbb{C}^{\ast}, \yY(d)) \leftarrow 
	\mathrm{Map}(\Theta, \yY(d)) \to \yY(d)
	\end{align*}
where $\Theta=[\mathbb{A}^1/\mathbb{C}^{\ast}]$, and the 
above arrows are induced by $\{0\}/\mathbb{C}^{\ast} \in \Theta$, 
$1 \in \Theta$, respectively
(see~\cite[Theorem~1.4.8]{Halpinstab}). 
	\end{rmk}
In the diagram (\ref{dia:loci}), 
the function $w^{\lambda \ge 0}$ is 
a defined to be the pull-back of $w$ by $p_{\lambda}$, which 
uniquely descends to a function $w^{\lambda=0}$. 
Since $p_{\lambda}$ is proper (as any fiber of $p_{\lambda}$ is 
a closed subscheme of the partial flag variety $\GL(V)/\GL(V)^{\lambda \ge 0}$), 
the following functor is well-defined 
\begin{align}\label{funct:p}
	p_{\lambda \ast}q_{\lambda}^{\ast}
	\colon \MF^{\C}(\yY(d)^{\lambda=0}, w^{\lambda=0}) \to 
	\MF^{\C}(\yY(d), w). 
	\end{align}
See~\cite[Section~3]{MR3270588} for the above functors of 
the categories of factorizations. 

We take the following special choice for $\lambda$
\begin{align*}
	\lambda(t)=(t, 1, \ldots, 1). 
	\end{align*}
Then $\dim V^{\lambda=0}=d-1$ and 
$\GL(V)^{\lambda=0}=\C \times \GL(V^{\lambda=0})$, so that we have 
\begin{align*}
	\yY(d)^{\lambda=0}=B\C \times \yY(d-1). 
	\end{align*}
We have the decomposition 
\begin{align*}
	\MF^{\C}(\yY(d)^{\lambda=0}, w^{\lambda=0})
	=\bigoplus_{j\in \mathbb{Z}}\oO_{B\C}(j) \boxtimes \MF^{\C}(\yY(d-1), w)
	\end{align*}
where
$\oO_{B\C}(j)$ is the $\C$-representation of weight $j$, 
and each direct summand is equivalent to 
$\MF^{\C}(\yY(d-1), w)$. 

It is easy to see that, when $X$ is a point, 
the stack $\yY(d)^{\lambda \ge 0}$ is isomorphic to 
the moduli stack of short exact sequences 
of $Q_{a, b}$-representations (see Remark~\ref{rmk:quiver})
\begin{align*}
	0 \to R'' \to R \to R' \to 0
	\end{align*}
where $R$ has dimension vector $(1, d)$
and $R''$ has dimension vector $(0, 1)$. It is straightforward 
to extend the above statement 
for an arbitrary $X$. 
Here we give some more details: 
\begin{lem}\label{lem:dia:P}
For $T \to X$, the $T$-valued points of the stack 
$\yY(d)^{\lambda \ge 0}$ consist of diagrams 
\begin{align}\label{ex:Tvalue}
	\xymatrix{
&0 \ar[r] \ar[d] & \eE_T^0 \ar@{=}[r] \ar[d]^-{\alpha} & \eE_T^0\ar[d]^-{\alpha'} & \\
0 \ar[r] &\pP'' \ar[r] \ar@{=}[d] & \pP \ar[r] \ar[d]^-{\beta} & \pP' \ar[d]^-{\beta'}
\ar[r] & 0 \\
&0 \ar[r] & \eE_T^{-1} \ar@{=}[r]  & \eE_T^{-1} &
}
	\end{align}
%\begin{align}\label{ex:Tvalue}
%	\xymatrix{
%		& 0 \ar[d] & \\
%		0\ar[r] \ar[d] & \pP'' \ar[r] \ar[d] & 0\ar[d] \\
%		\eE_T^0 \ar[r]^-{\alpha} \ar@{=}[d] & \pP \ar[r]^-{\beta} \ar[d] & \eE_T^{-1}\ar@{=}[d] \\
%		\eE_T^0 \ar[r]^-{\alpha'} & \pP' \ar[r]^-{\beta'} \ar[d] & \eE_T^{-1} \\
%		& 0 &
%	}
%\end{align}
where the middle horizontal sequence is an exact sequence of vector bundles on $T$
such that $\rank(\pP'')=1$, $\rank(\pP')=d-1$. 
The morphisms $p_{\lambda}$, $q_{\lambda}$ sends the above diagram to 
$(\pP, \alpha, \beta)$, $(\pP'', (\pP', \alpha',  \beta'))$ respectively. 
\end{lem}
\begin{proof}
We set $\zZ(d)^{\lambda \ge 0}=[Y(d)/\GL_X(V)^{\lambda \ge 0}]$
where $Y(d)$ is given in (\ref{def:Yd}). 
We have the factorization of the projection $\yY(d)^{\lambda \ge 0} \to X$
\begin{align*}
	\yY(d)^{\lambda \ge 0} \hookrightarrow 
	\zZ(d)^{\lambda \ge 0} 
	\to [X/\GL_X(V)^{\lambda \ge 0}] \to 
	[X/\GL_X(V)] \to X. 
	\end{align*}	
Here $\GL_X(V)^{\lambda \ge 0}$ and $\GL_X(V)$ act on $X$ trivially. 
For $T \to X$, giving its lift to 
$[X/\GL_X(V)]$ is equivalent to giving a vector bundle 
$\pP \to X$ of rank $d$. 
The fiber of $[X/\GL_X(V)^{\lambda \ge 0}] \to [X/\GL_X(V)]$
is $[\GL(V)/\GL(V)^{\lambda \ge 0}]$. 
Since $\GL(V)^{\lambda \ge 0}$ is the subgroup of 
$\GL(V)$ which preserves 
the one dimensional subspace $V^{\lambda >0} \subset V$, 
the stack $[\GL(V)/\GL(V)^{\lambda \ge 0}]$
is isomorphic to the projective space $\mathbb{P}(V)$ which parametrizes
one dimensional subspaces in $V$. 
Therefore giving a lift of $T \to [X/\GL_X(V)]$ to $[X/\GL_X(V)^{\lambda \ge 0}]$
is equivalent to giving a rank one vector subbundle $\pP'' \subset \pP$. 
By taking its cokernel, we obtain the exact sequence 
$0 \to \pP'' \to \pP \to \pP' \to 0$
of the middle horizontal sequence in (\ref{lem:dia:P}). 
Then giving its lift to $T \to \zZ(d)^{\lambda \ge 0}$ is equivalent to giving 
morphisms 
$\eE_T^0 \to \pP \to \eE_T^{-1}$. 
Since $V^{\lambda \ge 0}=V$ and $V^{\lambda \le 0}=V/V^{\lambda >0}$, the above lift 
$T \to \zZ(d)^{\lambda \ge 0}$
factors through $T \to \yY(d)^{\lambda \ge 0}$ if and only if 
$\pP \to \eE_T^{-1}$ factors through $\pP \twoheadrightarrow \pP' \to \eE_T^{-1}$. 
Therefore we obtain the lemma. 
	\end{proof}

The functor (\ref{funct:p}) gives 
the categorified Hall product 
\begin{align*}
	\ast \colon 
	\oO_{B\C}(j) \boxtimes \MF^{\C}(\yY(d-1), w)
	\to \MF^{\C}(\yY(d), w)
	\end{align*}
which is in fact induced by the stack of the diagrams (\ref{ex:Tvalue}). 
By the iteration, we also have the functor 
\begin{align}\label{cat:Hall}
	\ast \colon \oO_{B\C}(j_1) \boxtimes \cdots \boxtimes 
	\oO_{B\C}(j_l) \boxtimes \MF^{\C}(\yY(d-l), w) 
	\to \MF^{\C}(\yY(d), w). 
	\end{align}
In the case that $X$ is a point, the above product is a special 
case of categorical Hall products for quivers with super-potential 
(see~\cite[Section~3]{Tudor1.5}). 
The above product is their generalization to the family of moduli stacks
of representations of quivers. 
\subsection{Semiorthogonal decomposition}
We take a 
lexicographical order on $\mathbb{Z}^d$, i.e. 
for $m_{\bullet}=(m_1, \ldots, m_d) \in \mathbb{Z}^d$
and $m_{\bullet}'=(m_1', \ldots, m_d') \in \mathbb{Z}^d$, 
we write $m_{\bullet} \succ m_{\bullet}'$ if 
$m_i=m_i'$ for $1\le i\le k$ for some $k \ge 0$ and 
$m_{k+1}>m_{k+1}'$. 
	For
	$j_{\bullet}=(j_1, j_2, \ldots, j_l)$ 
	and 
	$j_{\bullet}'=(j_1', j_2' \ldots, j_{l'}')$
	with $l, l' \le d$, 
	we 
	define $j_{\bullet} \succ j_{\bullet}'$ if 
	we have $\widetilde{j}_{\bullet} \succ \widetilde{j}_{\bullet}'$, 
	where 
	$\widetilde{j}_{\bullet}$ is defined by 
	\begin{align}\label{jtilde}
		\widetilde{j}_{\bullet}=(j_1, j_2, \ldots, j_l, -1, \ldots, -1) \in \mathbb{Z}^d. 
	\end{align}
We recall that $(a, b, \delta)$ is defined in (\ref{abdelta}), 
and $\chi_0$ is the determinant character (\ref{det:chi0})
which determines a line bundle on $\yY(d)$. 
Below, we also assume that $\delta \ge 0$, i.e. $a\ge b$. 
By abuse of notation, we use the same symbol $\chi_0$ 
for the line bundle on $\yY(d')$ for any other $d'$ 
defined by the determinant character
on $\GL(\mathbb{C}^{d'})$. 

\begin{prop}\label{thm:hall}
	For $0\le j_1 \le \cdots \le j_l \le \delta-l$, 
	the categorified Hall product (\ref{cat:Hall})
	restricts to the fully-faithful functor
	\begin{align}\label{FF:W}
		\ast \colon 
		\oO_{B\C}(j_1) \boxtimes \cdots \boxtimes \oO_{B\C}(j_l)
		\boxtimes (\mathbb{W}_b(d-l)\otimes \chi_0^{j_l})
		\to \mathbb{W}_a(d)
		\end{align}
	such that, by setting 
	$\cC_{j_{\bullet}}$ to be the essential image
	of the above fully-faithful functor, 
	we have the semiorthogonal decomposition 
	\begin{align*}
		\mathbb{W}_a(d)=
		\langle \cC_{j_{\bullet}}:  0\le l\le d, 
		j_{\bullet}=(0\le j_1 \le \cdots \le j_l \le \delta-l)\rangle.  
		\end{align*}
	Here $\Hom(\cC_{j_{\bullet}}, \cC_{j_{\bullet}'})=0$ 
	for $j_{\bullet} \succ j_{\bullet}'$. 
		\end{prop}
	\begin{proof}
		The proposition is given in~\cite[Corollary~4.22]{Toconi}
		when $X$ is a spectrum of a complete local ring 
		and there is no auxiliary $\C$-action. The argument applies 
		verbatim with an auxiliary $\C$-action. 
		The categorified Hall products are defined globally, 
		and they have right adjoints by the same proof in~\cite[Lemma~6.6]{Toconi}. 
		Therefore in order to show that (\ref{FF:W}) is fully-faithful and 
		forms a semiorthogonal decomposition, it is enough to check 
		these properties formally locally 
		on $X$ (see the arguments of~\cite[Proposition~6.9, Theorem~6.11]{Totheta}
		or the last part of~\cite[Theorem~5.16]{Toconi}). 
		Therefore the proposition holds. 
		\end{proof}
	
	The following is the main result in this paper: 
	\begin{thm}\label{cor:sod}
		Suppose that $\delta \ge 0$. Then there 
		is a semiorthogonal decomposition of the form 
	\begin{align*}
		D^b(\mathbf{Quot}_{X, d}(\mathscr{G}))=
		\left\langle \binom{\delta}{i}\mbox{-copies of }
		D^b(\mathbf{Quot}_{X, d-i}(\mathscr{H})) : 
		0\le i\le \mathrm{min}\{d, \delta\}  \right\rangle. 		
	\end{align*}	
		\end{thm}
\begin{proof}
	In Proposition~\ref{thm:hall}, each semiorthogonal 
	summand $\cC_{j_{\bullet}}$ is equivalent to 
	$\mathbb{W}_b(d-l)$. 
	Therefore the corollary follows from Proposition~\ref{thm:hall}
	together with Lemma~\ref{lem:window} and equivalences (\ref{equiv:quot}). 
	\end{proof}
		\bibliographystyle{amsalpha}
\bibliography{math}

\vspace{5mm}

Kavli Institute for the Physics and 
Mathematics of the Universe (WPI), University of Tokyo,
5-1-5 Kashiwanoha, Kashiwa, 277-8583, Japan.

\textit{E-mail address}: yukinobu.toda@ipmu.jp
\end{document}